\documentclass[12pt]{article}
\usepackage{amsmath,amssymb,amsthm,latexsym,euscript,amscd, authblk}
\usepackage[english]{babel}
\usepackage[english]{babel}
\usepackage[all]{xy}
\RequirePackage{graphicx}
\usepackage{tikz}
\usepackage{easybmat}

\usepackage{xcolor}
\usepackage{thmtools,hyperref,cleveref}
\DeclareMathOperator{\Hom}{Hom}
\DeclareMathOperator{\Mat}{Mat}
\DeclareMathOperator{\Aut}{Aut}
\DeclareMathOperator{\Span}{span}
\DeclareMathOperator{\Ima}{Im}
\DeclareMathOperator{\Ker}{Ker}
\renewcommand{\P}{\mathbb{P}}
\newcommand{\id}{\mathrm{id}}
\newcommand{\G}{{\bf G}}
\newtheorem{Theorem}{Theorem}
\newtheorem{Proposition}[Theorem]{Proposition}
\newtheorem{Corollary}[Theorem]{Corollary}

\newtheorem{Question}[Theorem]{Question}
\newtheorem{Remark}[Theorem]{Remark}
\theoremstyle{definition}
\newtheorem{Example}[Theorem]{Example}
\newtheoremstyle{case}{}{}{}{}{}{:}{ }{}
\theoremstyle{case}
\newtheorem{case}{Case}

\title{On equivalence classes of homotopes of algebras and trilinear forms}

\author[1]{S.V. Guminov\thanks{sergey.guminov@gmail.com} }

\author[2, 3]{I.Yu. Zhdanovskiy\thanks {ijdanov@mail.ru}}

\affil[1]{National Research University High School of Economics}
\affil[2]{Moscow Institute of Physics and Technology, Laboratory AGHA}
\affil[3]{National Research University High School of Economics, Laboratory of Algebraic Geometry}
\usepackage{tikz}
\usetikzlibrary[patterns]

\numberwithin{Proposition}{section}
\numberwithin{Corollary}{section}
\numberwithin{Lemma}{section}
\numberwithin{Example}{section}
\numberwithin{Conjecture}{section}
\numberwithin{Problem}{section}
\numberwithin{Question}{section}
\numberwithin{Remark}{section}
\numberwithin{Theorem}{section}

\begin{document}

\maketitle

%\section{Introduction}
%\setcounter{Theorem}{0}
%\numberwithin{Proposition}{section}
%\numberwithin{Corollary}{section}
%\numberwithin{Lemma}{section}
%\numberwithin{Example}{section}
%\numberwithin{Conjecture}{section}
%\numberwithin{Problem}{section}
%\numberwithin{Question}{section}
%\numberwithin{Remark}{section}
%\numberwithin{Theorem}{section}
\begin{abstract}
    A homotope, or a mutation, of a $k$-algebra is a new algebra with the same underlying space, but with the multiplication law dependent on the multiplication law of the original algebra. In this paper, we show that a generic finite-dimensional algebra of dimension greater than 3 has infinitely many non-isotopic homotopes, and that, more generally, a similar result is true for generic trilinear forms. We also study a particular class of homotopes called $\Delta$-homotopes, where $\Delta$ is an element of the algebra, and show that there are algebras with infinitely many non-isomorphic homotopes even under some additional assumptions, such as the associativity of the algebra or $\Delta$ being well-tempered. 
\end{abstract}
\section{Introduction}

Non-associative algebras play an important role in modern algebra and its applications to areas such as geometry, quantum physics, and integrable systems. Lie algebras and Jordan algebras are the best-known examples of such algebras. Even though Lie algebras and Jordan algebras are very well-studied objects, general non-associative algebras are far from being fully understood.

Albert and Bruck (see \cite{Al}, \cite{Br}) established and studied the notions of {\it homotopy} and {\it isotopy} of non-associative algebras as a generalization of homomorphism and isomorphism  respectively. Namely, for a pair of $k$-algebras $A,\ B,$ a {\it homotopy} from $A$ to $B$ is a triple of $k$-linear maps $f_1, f_2, f_3:\ A\to B$ such that
\begin{equation}
\label{iso}
f_1(a) \cdot f_2(a')= f_3(a \times a')
\end{equation} 
for any $a,a' \in A$, where $\times$ and $\cdot$ are multiplication laws in $A$ and $B$ respectively, and $k$ is a field. If each one of these maps is bijective, the triple $(f_1,f_2,f_3)$ is called an {\it isotopy}. If the $k$-algebras $A$ and $B$ are isotopic, we will say that $A$ is an {\it isotope} of $B$ (or $B$ is an isotope of $A$). These notions enable one to get new structural results on non-associative algebras. 
For example, Bruck (see \cite{Br}) proved the structural theorem for finite-dimensional algebras and Malcev (\cite{Ma}) proved that an arbitrary algebra can be embedded into a suitable isotope of an associative algebra.

%For example, Bruck (see \cite{Br}) proved that a suitable isotope of a finite-dimensional algebra possesses a composition series whose factors are isotopically simple algebras. Recall that algebra is  isotopically simple if every isotope is simple algebra.
%Also, Malcev (see \cite{Mal}) proved the
%theorem on the embedding of an arbitrary algebra into a suitable isotope of an associative algebra.
%Note that there are many papers studying isotopes of alternative algebras and Jordan algebras (see \cite{}, \cite{}). Also, there is a notion of isotopy of loops and its geometric interpretation (see \cite{}). 

%The notion of homotopy of algebras is a natural generalization of isotopy. Assume that $A$ and $B$ are isomorphic as vector spaces over $k$. We will say that there is a homotopy from $(A, \times)$ and $(B, \cdot)$ if there are $k$-linear maps (not necessarily bijective) $f_1, f_2, f_3: A \to B$ such that (\ref{iso}) holds. 

The main objects of study of this paper will be \textit{homotopes} of finite-dimensional algebras. Given a triple of $k$-linear maps $f_1, f_2, g:\ A\to A$, we can define a new multiplication $\times$ on the underlying $k$-vector space of $A$ by setting 
\begin{equation}
\label{homotope}
a \times a'=g(f_1(a) \cdot f_2(a')).
\end{equation} Let $B$ denote the resulting algebra, which we call a {\it homotope} of $A$. If $g$ is invertible, $(f_1,f_2,g^{-1})$ is a homotopy from $A$ to $B$. We will later show how this definition arises naturally when one considers the multiplication law of an algebra as a trilinear form.

A {\it$\Delta$-homotope} is a special case of the above definition. Given a  $k$-algebra $A$ and an element $\Delta\in A$, we may construct a homotope by taking $f_1(a)=a\cdot\Delta,\ f_2=f_3=\id.$ We then call the resulting algebra a  \textit{left $\Delta$-homotope} of $A$. Similarly, we can define the \textit{right $\Delta$-homotope} by setting $f_1=f_3=\id, f_2(a)=\Delta\cdot a.$ Note that if the original algebra $A$ is associative, the multiplication laws we get, which are $a\times a'=(a\cdot\Delta)\cdot a'$ and $a\times a'=a\cdot(\Delta\cdot a')$ respectively, coincide and are associative. We then refer to the resulting associative algebra as simply a {\it $\Delta$-homotope} and denote it by $A_\Delta$. This special case, also known as a mutation of an algebra, and its generalizations appeared in connection with Jordan algebras and alternative algebras (see, for example, \cite{Ja}, \cite{McC}). 

We may now consider these definitions in the context of unital associative algebras. It is well-known that the notion of isotopy is not far away from isomorphism for the case of unital associative algebras: a unital algebra isotopic to a unital associative algebra is necessarily isomorphic to it. %if the maps $f_2,\ f_3$ of an isotopy preserve the identity, then $f_1=f_2=f_3$ and this map is an isomorphism. 
On the other hand, $\Delta$-homotopes are seldom isomorphic to the original algebra. If $A$ is a unital associative algebra, $\Delta \in A$, then the $\Delta$-homotope ${A}_{\Delta}$ is associative, but unital if and only if $\Delta$ is invertible in $A$. By adjoining a unit to the algebra $A_{\Delta}$, we get
a unital associative algebra $\widehat{A}_{\Delta}$, which we will call the {\it unital} or {\it augmented homotope} of $A$.

Despite the simplicity of the definition, the theory of representations of unital homotopes has deep connections with advanced modern constructions in homological algebra. One can find in \cite{BZ1} the connection between the representation theory of unital homotopes and topics such as discrete harmonic analysis on graphs, the description of orthogonal decompositions of simple Lie algebras into a sum of Cartan subalgebras, the classification of mutually unbiased bases (for more details, see \cite{BZ2}, \cite{KZ1}, \cite{KZ2}, \cite{KZ3}), which have wide application in quantum mechanics and quantum information theory, the description of perverse sheaves on stratified topological spaces, the representation theory of quasi-hereditary algebras,
and more. Also, if the element $\Delta \in A$ is well-tempered, i.e. satisfies some natural conditions (see for details \cite{BZ1}, \cite{Z} or Section \ref{Delta_homotopes}), then the construction of the unital $\Delta$-homotope gives us an example of recollement of abelian categories (see \cite{Psa}, \cite{Aus}). 

Let us consider a fixed finite-dimensional (associative or non-associative) algebra $A$. It is then natural to ask the following question.
\begin{Question}
Are there infinitely or finitely many homotopes of $A$ up to some equivalence?
\end{Question}
We will study this question in the following settings:
\begin{itemize}
    \item{$A$ is a general non-associative algebra, and homotopes are equivalent if and only if they are isotopic,}
    \item{$A$ is a fixed algebra, and homotopes are equivalent if and only if they are isomorphic.}
\end{itemize}

In the first case, we study the question using methods of algebraic geometry. The key idea is to assign to a fixed algebra a determinantal hypersurface. This assignment is an invariant of isotopy up to projective isomorphism: if two algebras are isotopic, then their determinantal hypersurfaces are projectively isomorphic. We show that determinantal hypersurfaces of the homotopes of a general algebra of dimension greater than 3 are generally non-isomorphic by considering the linear sections of such hypersurfaces up to a projective isomorphism. Varieties with projectively isomorphic hyperplane sections have been extensively studied, and the complete classification of such varieties of dimensions 1 and 2 is known. Our main references for this classification are the works of Pardini \cite{pardini_remarks_1994}, \cite{ballico_1992} and L'vovsky \cite{Lvovsky_1994}.

In the second case, we construct examples of both associative and non-associative algebras with infinitely many non-isomorphic $\Delta$-homotopes. Furthermore, we construct an associative algebra with infinitely many non-isomorphic $\Delta$-homotopes for a special case when $\Delta$ is a well-tempered element.

Our article is organized as follows. %The Section \ref{} is devoted to notation and preliminary remarks. 
In Section \ref{Homotopes_non_associative}, we study non-associative algebras and generalize our result to the case of general trilinear forms. Section \ref{Delta_homotopes} is devoted to the study of $\Delta$-homotopes of associative algebras.

{\bf Acknowledgements.}
We are grateful to Ilya Karzhemanov for fruitful discussions.
The study has been funded within the framework of the HSE University Basic Research Program. The second author was partially supported by the RFBR grant project 18-01-00908.

\subsection{Notation}
We will be working over an algebraically closed field $k$ of characteristic zero.  We will write $\P^n$ instead of $\P^n_k$ for the $n$-dimensional projective space over $k,$ considered as an algebraic variety.  %Furthermore, $\P^r\subset\P^n$ will always denote an $r$-dimensional linear subspace of $\P^n.$

We will use the same notation for a vector space and the corresponding affine algebraic variety. As an example, $X\subset V$, $V$ a finite-dimensional vector space, is a shorthand for $X\subset {\rm Spec}\ S^{\bullet}(V^*)$, with $V$ naturally identified with the closed points of the scheme ${\rm Spec}\ S^{\bullet}(V^*).$

\section{Homotopes of a general finite-dimensional non-associative algebra.}
\label{Homotopes_non_associative}

In this section we will study the following question:
\begin{Question}
For a general non-associative algebra $A$, are there infinitely or finitely many homotopes of $A$ up to isotopy?
\end{Question}

\subsection{Determinantal variety of a tensor.}
Consider three vector spaces $V_1, V_2, V_3$ of dimension $d_1, d_2, d_3$ respectively. For any vector space $V$ we denote by ${\mathbb P}(V)$ the projectivization of $V$. To a tensor $m \in V_1 \otimes V_2 \otimes V_3$ we can associate a determinantal variety $X(m)$ in one of the ${\mathbb P}(V_i^*)$. To simplify notation, and since that is the case we will be working with later on, in this section we will, without loss of generality, consider the case $i=1$.

 Fix some isomorphism $V_1\otimes V_2 \otimes V_3 \cong {\rm Hom}_k(V^*_1, V_2\otimes V_3)$. %A tensor $m\in V_1\otimes V_2 \otimes V_3$ is then the same thing as a linear morphism $m: V^*_1 \to V_2 \otimes V_3.$%, which then induces a morphism $m: \P(V^*_1) \to \P(V_2 \otimes V_3).$
%It is easy that $V_j \otimes V_k = {\rm Hom}_k(V^*_j, V_k)$. 
Recall that a tensor $t \in V_2 \otimes V_3$ has rank $r$ if there is a decomposition $t = \sum^r_{s=1} u_s \otimes w_s$ for some $u_s \in V_2, w_s \in V_3, s = 1,...,r$, and $r$, the number of terms in the sum, is minimal over all such decompositions. 

\iffalse
\begin{Remark}
In the sequel, we will frequently be working in the context of algebraic varieties and use the same notation for a vector space and the corresponding affine algebraic variety. As an example, $X\subset V$, $V$ a finite-dimensional vector space, is a shorthand for $X\subset {\rm Spec}\ S^{\bullet}(V^*)$, with $V$ naturally identified with the closed points of the scheme ${\rm Spec}\ S^{\bullet}(V^*).$  
\end{Remark} 
\fi

Denote by ${\cal D}_l(V_2, V_3) \subset V_2 \otimes V_3$ the subvariety consisting of tensors of rank not greater than $l\leqslant \min(d_2,d_3)$. We will make use of the well-known facts that ${\rm codim}_{V_2 \otimes V_3} {\cal D}_l(V_2,V_3)=(d_2-l)(d_3-l)$ and that the singular locus $\text{Sing }{\cal D}_l(V_2,V_3)$ is exactly ${\cal D}_{l-1}(V_2,V_3)\subset{\cal D}_{l-1}(V_2,V_3)$. (cf. \cite{Eisenbud_linear}). %It is a cone, so we may also consider its projectivization in $\P(V_2 \otimes V_3)$, which we also denote ${\cal D}(V_2, V_3)$. 
We set $l=\min(d_2,d_3)-1$ and define $X(m) \subseteq V^*_1$ as the pullback of ${\cal D}_l(V_2, V_3)$ along $m: V^*_1 \to V_2 \otimes V_3$. It is an affine cone. We also consider its projectivization $\hat{X}(m)=\P(X(m))$. Of course, if $m(V^*_1) \subseteq {\cal D}_l(V_2,V_3)$ then $\hat{X}(m) = {\mathbb P}(V_1)$ and vice versa. 

\iffalse

Note that $${\rm codim}_{\P(V_2 \otimes V_3)} {\cal D}(V_2,V_3) = \max(d_2,d_3)-\min(d_2,d_3)+1,$$ hence
$${\rm codim}_{\mathbb P(V_2)}\hat{X}(m) \le \max(d_2,d_3)-\min(d_2,d_3)+1.$$
\fi

If $m: V_1^*\to V_2\otimes V_3$ is injective, $\hat{X}(m)$ may also be described as the pullback of $\P({\cal D}_l(V_2, V_3)),$ and then for a general such $m$ the singular locus of $\hat{X}(m)$ will be the pullback of the singular locus $\P({\cal D}_{l-1}(V_2, V_3))$ and have the expected codimension $d_2+d_3-2l+1$.

We have the action of ${\bf G} = {\rm GL}(V_1)\times {\rm GL}(V_2)\times {\rm GL}(V_3)$ on $V_1\otimes V_2\otimes V_3$ defined as follows. $\forall {\bf g}=(g_1,g_2,g_3)\in {\bf G},\ v_i\in V_i,\ i=1,2,3,$ we define ${\bf g}\cdot (v_1\otimes v_2\otimes v_3)=(g_1(v_1)\otimes g_2(v_2)\otimes g_3(v_3))$ and extend this action to arbitrary tensors by linearity. Note that when we consider $m\in V_1\otimes V_2\otimes V_3$ as a linear morphism $V_1^*\otimes V_2^*\to V_3$ by fixing some isomorphism $V_1\otimes V_2 \otimes V_3 \cong {\rm Hom}_k(V^*_1\otimes V_2^*, V_3)$, we obtain that $({\bf g}\cdot m)(v^1\otimes v^2)=(g_3\circ m)(g_1^*(v^1)\otimes g_2^*(v^2))$ $\forall v^1\in V_1^*,\ v^2\in V_2^*$, where $g_i^*: V_i^* \to V_i^*$ are the maps adjoint to $g_i$. This generalizes the notion of isotopy, so for a pair of tensors $m_1, m_2\in V_1^*\to V_2\otimes V_3$ we will say that $m_1$ is \textit{isotopic} to $m_2$ and write $m_1 \sim m_2$ if there is an element ${\bf g} \in {\bf G}$ such that ${\bf g}\cdot m_1 = m_2$.
\begin{Proposition}\label{PROP_isotopic_isomorphic_determinantal}
If $m \sim m'$, then $X(m), X(m')\subset V_1^*$ are linearly isomorphic and $\hat{X}(m), \hat{X}(m')\subset \P(V_1^*)$ are projectively isomorphic.
\end{Proposition}
\begin{proof}
Let ${\bf g} = (g_1,g_2,g_3)$ be the element of ${\bf G}$ such that ${\bf g}\cdot m = m'$.
Since the rank of a pure tensor is invariant under the action of $\G$, ${\cal D}_l(V_2,V_3)$, $l=\min(d_2,d_3)-1,$ is invariant under the action of $\G$. When we consider $m$ as a morphism $V_1^*\to V_2\otimes V_3$, the action of $g_1$ on $V_1$ corresponds to the action $g_1\cdot m=m\circ g_1^*,$ $g_1^*:V^*_1\to V^*_1.$ Then the pullback of ${\cal D}_l(V_2,V_3)$ along ${\bf g}\cdot m$ is the same as the pullback along $m\circ g_1^*$, so $g_1^*$ induces a linear isomorphism between $X(m), X(m')\subset V_1^*$. Since $g_1^*$ is an isomorphism, it also induces a projective automorphism of $\P(V_1^*)$, which gives us a projective isomorphism of $\hat{X}(m)$ and $\hat{X}(m')$.

%If we pickup the bases of $V_i$, we can identidy the space ${\rm Hom}_k(V^*_2,V_3)$ with the space of rectangular matrices ${\rm Mat}_{d_3,d_2}(k)$. Action of ${\rm GL}(V_2) \times {\rm GL}(V_3)$ on ${\rm Mat}_{d_3,d_2}(k)$ is defined by the formula: $(g_2,g_3)(M) = g_3Mg^{-1}_2, M \in {\rm Mat}_{d_3,d_2}(k)$. Thus, images of $m^1_1$ and $m^1_2$ satisfy the relation: $Im(m^1_1) = g_3 Im(m^1_2) g^{-1}_2$. 
\end{proof} 

\subsection{Main result.}
Let $(V,m)$ denote a $k$-algebra with the underlying $k$-vector space $V$ and the multiplication law $m\in \Hom(V\otimes V,V)\cong V^* \otimes V^* \otimes V$. To apply the above construction to the study of algebras and their homotopes, we note that if $m, m' \in V^* \otimes V^* \otimes V,$ then $m \sim m'$ is equivalent to the algebras $(V,m)$ and $(V,m')$ being isotopic: $\bf{g}=(g_1,g_2,g_3)$ is such that $\bf{g}\cdot m=m',$ if and only if $(g_1^*,g_2^*,g_3^{-1})$ is an an isotopy from $(V,m)$ to $(V,m')$.

%In this case ${\cal D}_{d-1}(V^*,V)$ is a hypersurface of degree $d = {\rm dim}_k V$. 

Our goal in this section will be to prove the following theorem.
\begin{Theorem}\label{THM-main}
Let $A=(V,m)$ be a general $d-$dimensional $k-$algebra over an algebraically closed field $k$ of characteristic 0. Then there are infinitely many homotopes of $A$ up to isotopy if and only if $d\geqslant 4$.
\end{Theorem}

The theorem will follow from a series of propositions. We begin with the 'if' direction, the case when $d\geqslant 4$. Given a multiplication law $m$, our plan is to now show that under the construction $\hat{X}(m)$ described above, with $\hat{X}(m)$ being a subset of the projectivization of the dual space of the second term in $V^*\otimes V^*\otimes V$, there is a family of homotopes of $A=(V,m)$ corresponding to projective cones over plane sections of the determinantal variety $\hat{X}(m)$, and that these sections can not be generally isomorphic.

\begin{Proposition}\label{PROP_homotope_cone}
Let $A=(V,m)$ be a finite-dimensional $k-$algebra.Let $A'=(V,m')$ be the homotope of $A$ obtained by applying the triple $(f,\id,\id)$. Then $\hat{X}(m')$ is projectively isomorphic to a projective cone with basis $\P(\Ima f\cap X(m))\subset\P(\Ima f)\subset \P(V)$ and apex $\P(\Ker f)\subset \P(V).$
\end{Proposition}
\begin{proof}
Let $m'$ be the multiplication law of the homotopic algebra, i.e. $m'(a,b)=m(f(a),b)$. As in \autoref{PROP_isotopic_isomorphic_determinantal}, we see that the affine variety $X(m')$ is the pullback of $X(m)$ along $f$. We are interested in the case when $f$ has a kernel, since otherwise $(f,\id,\id)$ is an isotopy. But then the vector space $V$ splits noncanonically as $\Ker f\oplus V'$, with $f|_{V'}$ injective, which gives us an isomorphism of affine algebraic varieties $X(m')\cong(\Ima f\cap X(m))\times \Ker f,$ and then the projectivization $\hat X(m')$ is a projective cone with basis $\P(\Ima f\cap X(m))\subset\P(\Ima f)\subset \P(V)$ and apex $\P(\Ker f)\subset \P(V).$ 
\end{proof}

Following \cite{pardini_remarks_1994}, we say that a variety $X\subset \P^r$ over an algebraically closed field of characteristic zero has \textit{projectively isomorphic hyperplane sections} if there is a non-empty open set of hyperplanes $U\subset(\P^r)^\vee$ such that $\forall H,H'\in U$ the corresponding hyperplane sections $H\cap X, H'\cap X$ are projectively isomorphic, i.e. there is an automorphism $\gamma$ of $\P^r$ such that $ \gamma(H)=H',\ \gamma(H\cap X)=H'\cap X$. The sections $H\cap X$ for $H\in U$ are then called generic. Theorem 2.12 of \cite{pardini_remarks_1994} classifies all projective surfaces with projectively isomorphic hyperplane sections, while curves with projectively isomorphic hyperplane sections were classified in characteristic 0 by L'vosky \cite{Lvovsky_1994} and in arbitrary characteristic by Ballico \cite{ballico_1992}.

\iffalse
\begin{Theorem}[Pardini 2.12]\label{THM_Pardini}
Let $S \subset \P^r$ be a surface of degree $d$. $S$ has projectively isomorphic hyperplane sections if and only if $S$ is one of the following :
\begin{itemize}
    \item[(i)] a cone;
\item[(ii)] the surface of tangents of a twisted cubic in $\P^3$ ($d = 4$);
\item[(iii)] the Veronese surface $\Sigma\subset \P^5$ ($d = 4$);
\item[(iv)] a normal rational ruled surface ($d = r - 1$);
\item[(v)] a projection of the Veronese surface  $\Sigma\subset \P^5$ from a point in the plane of a
conic $\Gamma \subset \Sigma$ ($d = 4$);
\item[(vi)] the projection of a normal rational ruled surface $S' \subset \P^{r+1}$ from a point in
the plane of an irreducible conic $\Gamma \subset S'$ ($d = r$);
\item[(vii)] the projection of a normal rational ruled surface $S' \subset \P^{r+1}$ from a point in
the plane of a reducible conic $\Gamma \subset S'$ ($d = r$);
\end{itemize}
Moreover, the group $\mathrm{PGL}(4)$ acts transitively on the set of surfaces of type
(v) and, for fixed $r$, the group $\mathrm{PGL}(r)$ acts transitively on the set of surfaces of type
(vi) and on the set of surfaces of type (vii).

\end{Theorem}
\fi

\begin{Theorem}[Ballico, Theorem 0.2; L'vovsky, Proposition 0.1]\label{THM_PIHS_curves}
Let $C$ be an integral non-degenerate projective curve in $\P^n_k,$ $k$ algebraically closed of characteristic 0. Then it has projectively isomorphic hyperplane sections if and only if $\deg C\leqslant n+1.$
\end{Theorem}

\begin{Proposition}\label{PROP_not_projectively_isomorphic_sections}
For a general $d$-dimensional $k-$algebra $A=(V,m)$, $d\geqslant 4$, the projective variety $\hat{X}(m)\subset \P^{d-1}$ has infinitely many projectively non-isomorphic  sections by lines $\P^1\subset\P^{d-1}$.
\end{Proposition}

\begin{proof}
Consider a $d$-dimensional algebra $A=(V,m)$, $m\in V^*\otimes V^*\otimes V$, $\dim V = d\geqslant 4$. In this case, $\hat{X}(m)$ is a hypersurface of degree $d$, and for a general $m$, as it was mentioned before, its singular locus has the expected codimension 3. By repeatedly applying the Bertini theorem, we see that there is a smooth plane section by a $\P^2\subset\P^{d-1}$, which is a plane curve denoted by $C.$ The degree of this curve is $d\geqslant4$, so by \autoref{THM_PIHS_curves} it doesn't have projectively isomorphic hyperplane sections, which means that the original variety $\hat{X}(m)$ has infinitely many projectively non-isomorphic 1-dimensional plane sections.
\end{proof}

From the above propositions we get the 'if' part of \autoref{THM-main}:

\begin{Proposition}\label{PROP_main_d>3}
Let $A=(V,m)$ be a general $d-$dimensional $k-$algebra over an algebraically closed field $k$ of characteristic 0. If $d\geqslant 4$, there are infinitely many homotopes of $A$ up to isotopy.
\end{Proposition}

\begin{proof}
By \autoref{PROP_not_projectively_isomorphic_sections}, $\hat{X}(m)$ has infinitely many projectively non-isomorphic plane sections by 1-dimensional subspaces. For any $\P^1\subset\P^{d-1}$ there is a rank 1 map $f$ such that $\P(\Ima f)$ is exactly said $\P^1\in\P^{d-1}$. By varying $f$, we obtain an infinite family of homotopes $A=(V,m')$ of $A=(V,m)$ with their corresponding varieties being cones over projectively non-isomorphic varieties. 

An elementary argument shows that such cones are themselves projectively non-isomorphic. If $s$ is the dimension of $\Ima f$, then any plane section of $\hat{X}(m')$ by a dimension $s-1$ subspace $\P^{s-1}\subset \P(V)$ disjoint from the apex $\P(\Ker f)\subset \P(V)$  is projectively isomorphic to $\P(\Ima f\cap X(m)).$ Since this condition holds generally, $\P(\Ima f\cap X(m))$ is the generic $(s-1)$-dimensional plane section of $\hat{X}(m')$. 

Since projectively isomorphic varieties must have projectively isomorphic generic plane sections, the determinantal varieties $X(m')$ of our homotopes are themselves projectively non-isomorphic. Since isotopic algebras have isomorphic corresponding determinantal varieties, there are infinitely many pairwise non-isomorphic homotopes in the infinite family we constructed above.

\end{proof}

Now we consider the case $d<4.$ The case $d\leqslant 1$ is trivial. In case $d=2$, it is known that there are only 7 orbits of the action of $\G=\text{GL}(2)^{\times 3}$ on $V^*\otimes V^*\otimes V$ \cite{Parfenov_2001}, so there are only finitely many $2-$dimensional algebras up to isotopy. The only remaining case is $d=3$.
\begin{Proposition}\label{PROP_d=3}
Let $A=(V,m)$ be a 3-dimensional $k-$algebra. Then it has only finitely many homotopes up to isotopy.
\end{Proposition}

\begin{proof}
First we note that we may construct a homotope using a triple $(f_1,f_2,f_3)$ by consecutively constructing homotopes using the triples $(f_1,\id,\id)$ and $(\id,f_2,\id)$ and $(\id,\id,f_3)$. It will be sufficient that up to isotopy there are finitely many homotopes with only one of the maps $f_i$ is not the identity. We first consider the case $(f_1,\id,\id)$. If $f_1$ is invertible there is nothing to prove. Assume now that the map $f_1$ is not invertible, so its rank is not greater than 2. Consider $m$ as a map from the first term in $V^*\otimes V^*\otimes V$, i.e. an element of $\Hom(V, V^*\otimes V).$ This way the multiplication law of the homotopic algebra $m'$ becomes the linear map $m'\in \Hom(V, \Hom(V,V))$, $m'=m\circ f_1$.

We can write $f_1$ as the composition $i\circ \pi:$

\[V\xrightarrow{\pi} W\xrightarrow{i} V,\] where $\text{rank } i=\text{rank } f_1$, $\text{dim }W=2$ and $\pi$ is a surjective map fixed once and for all for all choices of $f_1$. 
Then we can consider $m\circ i$ as an element of $W^*\otimes V^*\otimes V$, i.e. a tensor of type (2,3,3).  It was shown in \cite{Parfenov_2001} that there are only finitely many tensors of this type up to the action of $\text{GL}(W^*)\times\text{GL}(V^*)\times\text{GL}(V)$. Here $h^*\in\text{GL}(W^*)$ acts by composition with $h$ on the right.  Notice that since $\pi$ is surjective, for any automorphism $h$ of $W$ there is an automorphism $g$ of $V$ such that $h\circ\pi=\pi\circ g$. In other words, the action of $h^*\in\text{GL}(W^*)$ on $\tilde{m}\circ i$ in the composition $m\circ i\circ \pi$ may be realized through the action of some $g^*\in\text{GL}(V^*)$ on $m'\circ i\circ\pi. $ From this we have that for any representative $\hat{m}$ of the orbit of $m\circ i$ in $V^*\otimes W^*\otimes V$ we have $m'=m\circ f_1=m\circ i\circ \pi \sim \hat{m}\circ \pi.$ It follows that the the number of homotopes of $m$ up to isotopy in this case is no more than the number of tensors of type (2,3,3) up to isotopy, which is finite. 

The argument for the case $(\id,f_2,\id)$ is identical. For the case $(\id,\id,f_3)$ we consider $m$ and $m'$ as a elements of $\Hom(V\otimes V, V)$. Then we have that $m'=f_3\circ m,$ and we factorise $f_3$ as $i\circ \pi:$

\[V\xrightarrow{\pi} L\xrightarrow{i} V,\]
now with $\text{dim } L=2$, $i$ injective and $\text{rank } \pi=\text{rank } f_3.$ We then proceed analogously, lifting an automorphism of $L$ to an automorphism of $V.$
\end{proof}

\subsection{Generalization: homotopes of trilinear forms.}

The result of \autoref{THM-main} may be generalized to the case of trilinear forms, i.e. tensors in $V_1\otimes V_2\otimes V_3$ for a general triple of $k$-vector spaces $V_i,\ i=1,2,3$. Given a triple of linear maps ${\bf f}=(f_1,f_2,f_3)$ $f_i: V_i\to V_i,\ i=1,2,3,$ and $m =\sum_s v_{1,s}\otimes v_{2,s}\otimes v_{3,s}\in V_1\otimes V_2\otimes V_3,$ let ${\bf f}\cdot m=\sum_s f_1(v_{1,s})\otimes f_2(v_{2,s})\otimes f_3(v_{3,s}).$ We say that ${\bf f}\cdot m$ is a homotope of $m$ and $f$ is a \textit{homotopy}. Note that in the particular case when $m$ is a multiplication law of an algebra $A$, i.e. $m\in V^*\otimes V^*\otimes V,$ the homotope ${\bf f}\cdot m$ will be the multiplication law of the homotope of $A$ with multiplication defined by $a\times a'=f_3(f_1^*(a)\cdot f_2^*(a)),$ which motivates our choice of terminology. We may now ask whether there are generally infinitely many homotopic tensors up to isotopy. We shall now prove that this is indeed the case in sufficiently many dimensions.

\begin{Theorem}\label{THM-main-general}
Let $V_i,\ i=1,2,3$ be vector spaces over $k$, an algebraically closed field of characteristic 0. Let $d_i=\dim V_i$ and assume, without loss of generality, that $d_1\geqslant d_2\geqslant d_3$. Let  $m$ be a general element of $V_1\otimes V_2\otimes V_3$. Then there are infinitely many tensors homotopic to $m$ up to isotopy if and only if either $d_1\geqslant 4$ and $d_3\geqslant 3$, or $d_1\geqslant5,\ d_2\geqslant 4,\ d_3=2$.  
\end{Theorem}

\begin{proof}
The proof goes along the lines of the proofs of the previous section. 

We begin with the 'only if' direction. It is known that there are only finitely many tensors of type $(d_1,d_2,d_3)$ up to isotopy if and only if $(d_1,d_2,d_3)$ is  of the form $(n,m,1),\ n,m\in\mathbb{N},$ or $(n,m,2)$, $n\in\mathbb{N}, m\in\{2,3\}.$ It is easy to see that if the triple $(d_1,d_2,d_3)$ does not satisfy the inequalities in the statement, then each one of the triples $(d_1-1,d_2,d_3),$ $(d_1,d_2-1,d_3) $ and $(d_1,d_2,d_3-1)$  is of this form. Then we can proceed exactly as in the proof of \autoref{PROP_d=3}, where we used the finiteness of tensors of tensors of type $(2,3,3).$ 

Now we move on to the 'if' direction. When $d_3\geqslant 3$, we can apply the method of the previous section. The proof of \autoref{PROP_homotope_cone} repeated verbatim shows that for homotopies of the form ${\bf f}=(f,\id,\id)$ and $m'={\bf f}\cdot m$ $\hat{X}(m')$ is projectively isomorphic to a projective cone with basis $\P(\Ima f\cap X(m))\subset\P(\Ima f)\subset \P(V)$ and apex $\P(\Ker f)\subset \P(V).$ We will now prove a statement analogous to \autoref{PROP_not_projectively_isomorphic_sections}: for a general $m$ the projective variety $\hat{X}(m)\subset \P(V_1^*)$ has infinitely many projectively non-isomorphic plane sections by linear subspaces $L\subset\P(V_1^*)$, which will imply that there are infinitely many homotopes up to isotopy.

Assume first that $d_1\leqslant d_2d_3$. In this situation, for a general $m$ the corresponding linear map $m: V_1^*\to V_2\otimes V_3$ is injective, $\dim \hat{X}(m)=d_1-d_2+d_3-2,$ and $\dim \text{Sing}(\hat{X}(m))=d_1-2d_2+2d_3-5.$ We proceed case by case:

\begin{case} $d_2\geqslant4.$
 For a general linear subspace $\P^{2+d_2-d_3}\subset\P(V_1^*)$ and $C=\hat{X}(m)\cap \P^{2+d_2-d_3},$ \[\dim C=(d_1-d_2+d_3-2)+(2+d_2-d_3)-(d_1-1)=1\] and $C$ is irreducible and nonsingular.

The degree of $\hat{X}(m)$, and hence the degree of $C$, is $\binom{d_2}{d_3-1}$ (cf. \cite{HARRIS198471}). Let $r=d_2-d_3.$ $C$ is a curve in $\P^{2+r},$ so according to \autoref{THM_PIHS_curves}, if $\binom{d_2}{d_3-1}=\binom{d_2}{r+1}>3+r$ the curve $C$ does not have projectively isomorphic hyperplane sections.

If $r=0$, $\deg C=d_2\geqslant 4>3+r$. If $r>0,$ $\binom{d_2}{r+1}\geqslant \binom{r+3}{r+1}=(r+3)\left(1+\frac{r}{2}\right)>3+r$. In both cases $C$ does not have projectively isomorphic hyperplane sections.
\end{case}

\begin{case} $d_2=d_3=3.$
In this case,  $\dim \hat{X}(m)=d_1-2,$ and $\dim \text{Sing}(\hat{X}(m))=d_1-5.$ Then for a general linear subspace $\P^{3}\subset\P(V_1^*)$ and $S=\hat{X}(m)\cap \P^{3}$ \[\dim S=(d_1-d_2+d_3-2)+(2+d_2-d_3)-(d_1-1)=2,\] and $S$ is irreducible and nonsingular.

The degree of $\hat{X}(m)$, and hence the degree of $C$, is $\binom{d_2}{d_3-1}=3$, so $S$ is a smooth cubic hypersurface. We will now show that a general smooth cubic hypersurface in $P^3$ does not have projectively isomorphic hyperplane sections. Such a cubic is isomorphic to a blowup of $P^2$ at 6 general points. Choose 8 points of $P^2$ in general position. There is a pencil of cubic curves passing through these 8 points. Blow up 6 of these points to obtain a cubic surface $S\subset P^3$ and let $l$ be the line passing through the two remaining points. Let $C_1,\ldots, C_6$ be the 6 exceptional curves and $\pi:S\to\P^2$ the blow-up morphism.

 Suppose $S$ has projectively isomorphic hyperplane sections by hyperplanes $L\subset P^3.$ By dimension count, we see that for a general line $l\in P^3$ not lying on $S$ the sections by a general plane $L$ passing through $l$ are projectively isomorphic. Then all of these sections must intersect $C_1,\ldots, C_6$ and pass through the 3 points of intersection of $l$ with $S$. Applying $\pi$ then gives us a pencil of cubic curves passing through 9 points, which is exactly the pencil we began with. But for a general such pencil of cubic curves through 8 points, the members of the pencil are generally non-isomorphic, which contradicts our assumption that $S$ has projectively isomorphic hyperplane sections by planes $L\in P^3$.

\end{case}
Finally, let $d_1>d_2d_3.$ In this case for a general $m$ the corresponding linear map $m: V_1^*\to V_2\otimes V_3$ is of maximal rank $d_2d_3$, and hence there is a subspace $W\subset V_1^*$ of dimension $d_2d_3$ with the inclusion map $i: W\to V_1^* $ such that the composition $m\circ i$ is injective. Then the above argument works with $\hat{X}(m)$ replaced by its pullback along $i$. But this pullback is nothing more than a linear section of $\hat{X}(m)$ by a linear subspace of dimension $d_2d_3,$ so in this case  $\hat{X}(m)$ also has infinitely many projectively non-isomorphic sections by linear subspaces of dimension 1.

Only the case $d_1\geqslant5,\ d_2\geqslant 4,\ d_3=2$ remains. In this situation, the same method does not seem to apply. For instance, in the case $d_1=d_2$, the variety $\hat{X}(m)$ is zero-dimensional and there are no interesting hyperplane sections to speak of. Instead, we proceed directly.

First we reduce to the case $(d_1,d_2,d_3)=(5,4,2).$ Fix some subspaces $W_1\subset V_1,\ W_2\subset V_2,\ \dim W_1=5,\ \dim W_2=4$ and projections $p_i:V_i\to W_i.$ Assume we have already shown that there is an open set $U\in W_1\otimes W_2\otimes V_3$ such that for $m\in U$ there are infinitely many homotopes up to isotopy. Then the triple $(p_1,p_2,\id)$ is also a homotopy and, as a consequence, there are infinitely many homotopes up to isotopy for tensors in $(p_1\otimes p_2\otimes \id)^{-1}U$, which is an open subset of$V_1\otimes V_2\otimes V_3.$

Now let $(d_1,d_2,d_3)=(5,4,2).$ After fixing the bases of $V_i$, the tensor $m$ can be viewed as a pair of $5\times 4$ matrices $(A,B)$. Then for $G_i\in \text{GL}(V_i)$ the action takes the form $G_1\cdot(A,B)=(G_1A,G_1B), $ $G_2\cdot(A,B)=(AG_2^T,BG_2^T)$   and $G_3\cdot(A,B)=(\alpha A+\beta B, \gamma A +\delta D),$ where $G_3=\left(\begin{smallmatrix} \alpha\ \beta\\ \gamma\ \delta\end{smallmatrix}\right).$ Applying a homotopy then corresponds to applying the same formulas for not necessarily invertible matrices $G_i\in \Mat_i(k).$

For a generic tensor, both $A$ and $B$ are full-rank. Then, possibly after applying some isotopy, we may assume that this pair of matrices has the form

\[A=\left(
\begin{BMAT}(b){c}{c|c}
I_4\\ 0\end{BMAT}
\right)
,\ B=
\left(
\begin{BMAT}(b){c}{c|c}
\hat{B}\\b\end{BMAT}
\right),
\] where $I_4$ is the identity matrix of rank $d_2$, $b$ is a row vector of size 4, and $\hat{B}$ is a $4\times 4$ diagonal matrix with the $i$-th diagonal entry $\hat{B}_i$.

Define $G_1\in \text{Mat}_{5}(k)$ as

\[
G_1=
\left(
\begin{BMAT}(b){c|c}{c|c}
I_4 & u\\
0 & 0\end{BMAT}
\right),\] where $u$ is a column vector of size 4. Apply to $m$ the homotopy $(g_1,\id,\id)$, where $g_i$ is represented by the matrix $G_i$ in our chosen bases. Let $m'$ denote the homotopic tensor and let $(A'|B')$ be $m'$ written as a pair of matrices. Then

\[A'=G_1A=\left(
\begin{BMAT}(b){c}{c|c}
I_4 \\
0  \end{BMAT}
\right),\
B'=G_1B=\left(
\begin{BMAT}(b){c}{c|c}
\hat{B} + ub\\
0\end{BMAT}
\right).
\]

Consider the subvariety of $\P^1$ with homogeneous coordinates $[x:y]$ defined by the inequality $\text{rank }(xA'+yB')<4.$ It is clearly preserved up to projective isomorphism by the group action. It can also be defined by the equation $\det (xI_4+y(\hat{B}+ub^T))=0,$ and it consists of four points, counted with multiplicities.  We will now show that for a general matrix $B$ by varying $u$ we may obtain infinitely many configurations of 4 points.

Let $u_i,b^i,\ 1\leqslant i\leqslant 4$, be the entries of vectors $u,\ b$, and consider $\det (xI_4+y(\hat{B}+ub^T))=0,$ as a function of $u_i$. We find its derivative with respect to $u_i$ at the origin:

\[\left[\frac{\partial}{\partial u_i}\det (xI_4+y(\hat{B}+ub))\right](0)=yb^i\prod_{j\neq i}(x+y\hat{B}_{jj}).\]

A straightforward check shows that for a general $B$ (and hence $\hat{B}$ and $b$) these polynomials for $i\in\{1,2,3,4\}$ are linearly independent, so the image of the map $u\longmapsto \det (xI_4+y(\hat{B}+bu^T))$ has dimension 4. But the coefficient of $x^4$ in $\det (xI_4+y(\hat{B}+bu^T))$ is 1, so after passing to the projectivization we obtain a dense set of $\P^4$, the projective space of polynomials of degree 4 in two variables. That means that for fixed (but general) $B$ and $b$ a general degree 4 subvariety of $\P^1$ is defined by an equation of the form $\det (xI_4+y(\hat{B}+bu^T))=0$. Since such subvarieties are generally not projectively isomorphic, we have that by varying $u$ we obtain an infinite family of non-isotopic tensors $m'$. \end{proof}

\section{Isomorphism classes of $\Delta-$homotopes.}
\label{Delta_homotopes}

Now we turn our attention to $\Delta-$homotopes. In this section, we do not provide any results for a general algebra. The reason why we can't apply the methods of the previous sections in this setting is that while previously we had total freedom in choosing a homotopy, whether a given homotopy gives a $\Delta-$homotope or not heavily depends on the multiplication law of the original $k-$algebra. We also consider isomorphism classes instead of isotopy classes, because the unital associative case will be the most interesting for us, and in this case, the unity-preserving isotopies are exactly the isomorphisms. In other words, the question we will be concerned with in this section is as follows.

\begin{Question}\label{QST_delta_homotopes}
For a fixed finite-dimensional (non-associative, unital associative) $k$-algebra $A$, can there be infinitely many (left) $\Delta-$homotopes up to isomorphism? 
\end{Question}

\subsection{The non-associative case.}
It turns out that for the case of non-associative algebras, already in dimension 2 the answer to \autoref{QST_delta_homotopes} is positive.

Let $A$ be a two-dimensional $k$-algebra with the basis $\{e_1,e_2\}$ and the multiplication law 
$$e_1e_1=e_1,\ e_2e_1=e_2,\ e_1e_2=e_1,\ e_2e_2=e_1.$$ Let $\Delta(\lambda)=e_1+\lambda e_2,\ \lambda \in k$, and consider the left homotopes $B_\lambda=A_{\Delta(\lambda)}$. Then we claim that the algebras $B_\lambda$ are pairwise nonisomorphic for different $\lambda\in k$.
\begin{Theorem}
The 2-dimensional non-associative algebras $B_\lambda$ are pairwise nonisomorphic for different $\lambda\in k$.
\end{Theorem}
\begin{proof}
The multiplication law of $B_\lambda$ (which we do not denote by any special symbol in this section) is
$$e_1e_1=e_1+\lambda e_1\,\ e_1e_2=e_1+\lambda e_1,\ e_2e_1=e_2+\lambda e_1,\ e_2e_2=e_1+\lambda e_1.$$ Denote by $[\cdot,\cdot]$ the commutator of two elements, i.e. $[x,y]=xy-yx$. We have that $[e_1,e_2]=e_1-e_2.$ Consider the equation $x^2=x$ in $B_\lambda$. If $\lambda = -1,$ it has a unique non-zero solution $x=e_1-e_2$. Otherwise, if $\lambda\neq -1$ it has 3 non-zero solutions in $B_\lambda:$ $x_1=e_1-e_2,\ x_2=\frac{1}{1+\lambda}e_1,\ x_3=e_1+\frac{\lambda}{1+\lambda}e_2$. This shows that $B_{-1}$ is not isomorphic to any other $B_\lambda, \lambda\in k\setminus\{-1\}.$

If $\lambda\neq -1$, we also  have $[x_1,x_2]=-\frac{1}{1+\lambda}x_1,\ [x_1,x_3]=\frac{1+2\lambda}{1+\lambda}x_1,\ [x_2,x_3]=\frac{\lambda}{(1+\lambda)^2}x_1.$ Hence the system of equations 
$$\begin{cases} x_i^2=x_i,\ i=1,2,3, \\ [x_1,x_2]=-\frac{1}{1+\lambda}x_1,\\ [x_1,x_3]=\frac{1+2\lambda}{1+\lambda}x_1\end{cases}$$ has a non-zero solution in $B_\lambda $ and no solutions in $B_\mu$, $\mu\neq\lambda.$ Since an isomorphism must preserve the solutions of this system, the algebras $B_\lambda$ are pairwise nonisomorphic.
\end{proof}
\subsection{The associative case.}\label{SEC-Associative}
For the unital associative case we begin by establishing some conditions that guarantee that a pair of homotopes $A_\Delta$ and $A_{\Delta'}$ are isomorphic.

In this section, $\cdot$ will denote the multiplication in $A_\Delta$, $\times $ will denote the multiplication in $A_{\Delta'}$, and the multiplication in the original algebra $A$ will be denoted simply by concatenation. Let $U$ be the group of units of $A$. Assume that there are $u,v\in U$ such that $\Delta'=u\Delta v.$ Define $\psi: A_\Delta\to A_{\Delta'}$ by $\psi(x)=v^{-1}xu^{-1}$. It is an isomorphism of algebras, since $$\psi(x\cdot y)=\psi(x\Delta y)=v^{-1}x\Delta y u^{-1}=v^{-1}xu^{-1}u\Delta vv^{-1} y u^{-1}=v^{-1}xu^{-1}\Delta'v^{-1} y u^{-1}=\psi(x)\times \psi(y).$$

\begin{Example}
Let $A$ be the algebra of $n\times n$ matrices over the base field, i.e. $A=\Mat_n(k).$ Let $\Delta\in A$ be a rank $r$ matrix, then there are some matrices $U, V\in A$ such that $U\Delta V$ is a diagonal matrix with the first $r$ diagonal entries are equal to 1 and all other entries are zero. From the above considerations, we see that the rank $r$ completely defines the homotope $A_\Delta$ up to isomorphism, so we have $n+1$ isomorphism classes of $\Delta$-homotopes.
\end{Example}

Now assume that there is an automorphism $\varphi$ of $A$ such that $\varphi(\Delta)=\Delta'.$ Then the same $k$-linear map $\varphi$, viewed as a map $\varphi: A_\Delta\to A_{\Delta'},$ is an isomorphism of algebras, since we have $$\varphi(x\cdot y)=\varphi(x\Delta y)=\varphi(x)\varphi(\Delta)\varphi(y)=\varphi(x)\Delta'\varphi(y)=\varphi(x)\times \varphi(y).$$

\begin{Example}
Let $\Gamma$ be a quiver on two vertices $e_1, e_2$ and with two arrows $\alpha,\beta: e_1\to e_2$, and consider the quiver path algebra $A=k\Gamma$. Let $\Delta(\lambda)=\alpha+\lambda \beta$, $\lambda\in k\setminus {0}.$ Then $\Delta(\lambda),\ \Delta(\lambda')$ do not differ by an invertible element for $\lambda\neq\lambda'$, but the automorphism $\varphi$ of $A$ which maps $\beta$ to $\frac{\lambda'}{\lambda}\beta$ and is the identity on $\alpha$ and stationary paths takes $\Delta(\lambda)$ to $\Delta(\lambda')$ and, therefore, induces an isomorphism between $A_\Delta$ and $A_\Delta'.$ In this case $A$ also has finitely many isomorphism classes of $\Delta$-homotopes.
\end{Example}

From the above considerations, we have that both $U\times U^{op}$ and $\Aut A$ act on the set of $\Delta-$homotopes by isomorphisms. But also note that all the maps of the form $\psi(x)=uxu^{-1}$ for $u\in U$ are actually automorphisms of $A$, so it is enough to only consider the action of $U$ by multiplication of the left, and that for an automorphism $\psi$ we have $\psi(v^{-1}xu^{-1})=\psi(v)^{-1}\psi(x)\psi(u)^{-1}$. The following theorem collects the previous considerations.

\begin{Proposition}
The group $G=\Aut A\rtimes U$ acts on the set of $\Delta$-homotopes  of $A$ by isomorphisms.
\end{Proposition}

We note that there are isomorphisms of $\Delta$-homotopes which don't come from the action of the group $G=\Aut A\rtimes U$.

\begin{Example}
Consider the path algebra $k\Gamma$ of a quiver with $6$ vertices denoted $e_1,\ldots, e_6$ and $2$ arrows $e_i\to e_{i+1}$ for $i=1,\ldots,5$. There are $2^5=32$ paths of length $5$. The group of automorphisms of the path algebra is isomorphic to $\text{GL}(2)^5$, so has dimension $20.$ On the subspace generated by paths of length $5$ multiplying by a unit on the left may only change an element by some scalar factor: for $\Delta\in e_6k\Gamma e_1$ and a unit $u=\lambda e_6+r$, $e_6re_6=0$, $\lambda\neq 0$, we have  $u\Delta=\lambda \Delta.$ So just from dimension considerations the group $G$ can't have finitely many orbits in the set of $\Delta$-homotopes with $\Delta\in e_6k\Gamma e_1$. However, these homotopes are actually all isomorphic, and an explicit isomorphism $\varphi: A_\Delta\to A_{\Delta'}$  can be constructed by setting $\varphi(e_1)=e_1,\ \varphi(e_6)=e_6,\varphi(\Delta)=\Delta'$ and arbitrarily extending this to an invertible linear map.
\end{Example}

We now construct an example of a finite-dimensional algebra with infinitely many pairwise non-isomorphic homotopes. We start with the infinite family of 4-dimensional associative algebras described in \cite{gabriel1975}. Denote $R_\lambda=k\langle x,y \rangle/{(x^2,y^2,xy-\lambda yx)}.$ We have $R_\lambda\cong R_{\lambda'}$ if and only if $\lambda =\lambda'$ or $\lambda=\lambda^{-1}.$ To see this, note that in the algebra $R_{\lambda}$, $\lambda\neq -1$ the only elements satisfying $t^2=0$ are of the form $a_x x+a_{xy}xy$ or $b_yy+b_{xy}xy$, $a_x,a_y,a_{xy},b_{xy}\in k$, so the algebra $R_\lambda$ has a pair of square-zero generators, and any two such generators $z_1,z_2$ must satisfy $z_1z_2=\lambda^{\pm 1}z_2z_1$. The algebra $R_{-1}$ is unique in that it has a three-dimensional space of square-zero elements.

Now define $B=k\langle x,y,h_1,h_2\rangle/ {I},$ where the ideal of relations $I$ is generated by
\begin{itemize}
    \item $x^2,\ y^2,\ h_ih_j,\ i,j=1,2$
    \item $xh_1y-yh_2x,$
    \item all the words of length 3 except $xh_1y,yh_2x$.
\end{itemize}

Then $B$ is a 16-dimensional unital associative algebra. Let $\Delta(\lambda)=\lambda h_1+h_2$. We claim that the homotopes belonging to this infinite family are generally nonisomorphic.

\begin{Theorem}
The associative algebras $B_{\Delta(\lambda)}, B_{\Delta(\lambda')}$, $\lambda,\lambda'\in k$ are isomorphic if and only if $\lambda'=\lambda^{\pm 1}$.
\end{Theorem}
\begin{proof}
%As a $k$-vector space, $B\lambda$ is representable as the direct sum $\Span(x,y, xh_1y)\bigoplus N.$
To simplify notation, we will write $B_\lambda$ instead of $B_{\Delta(\lambda)}$ in the proof. Let $\hat{B}_\lambda$ be the unital associative algebra obtained by adjoining a unit to $B_\lambda$. As a vector space, we can write $\hat{B}_\lambda$ as some arbitrary direct sum of $\Span(1,x,y, xh_1y)$ and the linear span of all the other paths, which we denote by $N$. Note that $1$ here is the unit of $\hat{B}_\lambda$, not the element of $B_\lambda$ coming from the unit of $B$.

Now let  $\cdot$ denote the multiplication in $\hat{B}_{\lambda}$. Note that $$x\cdot x=y\cdot y=0,$$ $$x\cdot y-\lambda y\cdot x=x(\lambda h_1+h_2)y-\lambda y(\lambda h_1+h_2)x=\lambda xh_1y-\lambda yh_2x=0,$$so $x,y$ satisfy the defining relations of the algebra $R_\lambda$ defined previously. 

This means that $\hat{B}_\lambda=R_\lambda\bigoplus N$ as a $k$-vector space, and $R_\lambda\cdot R_\lambda\subset R_\lambda,$ $N\cdot N=0.$ Moreover, $N\cdot R_\lambda= R_\lambda\cdot N=N$, with the only non-zero products being with scalar multiples of $1\in R_\lambda$. All of this implies that if $\hat{B}_\lambda$ were isomorphic to $\hat{B}_{\lambda'},$ it would be possible to choose an isomorphism restricting to an isomorphism $R_\lambda\cong R_{\lambda'},$ which does not exist unless $\lambda=\lambda'$ or $\lambda={\lambda'}^{-1}.$ Finally, we notice that a pair of (not necessarily unital) algebras $A_1,\ A_2$ are isomorphic if and only if the algebras $\hat{A}_1,\ \hat{A}_2$ obtained by adjoining a unit are isomorphic as unital algebras.
\end{proof}

\subsection{The well-tempered case.}

Let $A$ be a unital associative algebra. An element $\Delta\in A$ is called \textit{well-tempered} if $A\Delta A=A$ and $A$ is projective as a left and right $\hat{A}_\Delta$-module. If $A$ is finite-dimensional over $k$, it is enough to check that $A\Delta A=A.$ It has been shown that $\Delta-$homotopes with $\Delta$ well-tempered have some nice properties (see \cite{BZ1}).  Nevertheless, we will now show that there can still be an infinite family of nonisomorphic $\Delta$-homotopes even for all $\Delta$ well-tempered.

 Consider $M=\Mat_2(A)$, the algebra of $2\times2$ matrices over $A.$ Let $\Lambda=\left(\begin{smallmatrix} 1\ 0\\ 0\ \Delta\end{smallmatrix}\right)$ for $\Delta\in A.$ This is a well-tempered element for any $\Delta$: it is easy to see that the two-sided ideal generated by the matrix $\left(\begin{smallmatrix} 1\ 0\\ 0\ 0\end{smallmatrix}\right)$ is equal to $M$ and that $$\left(\begin{smallmatrix} 1\ 0\\ 0\ 0\end{smallmatrix}\right)\Delta\left(\begin{smallmatrix} 1\ 0\\ 0\ 0\end{smallmatrix}\right)=\left(\begin{smallmatrix} 1\ 0\\ 0\ 0\end{smallmatrix}\right).$$
 
 Consider the homotope ${M}_\Lambda$, with multiplication law denoted by $\cdot$, so for $X, Y\in {M}_\Lambda\ $ $X\cdot Y=X\Lambda Y$, with the multiplication on the right being the regular matrix multiplication of $M$. Let $\hat{M}_\Lambda$ be the homotope with the adjoined unit, with the multiplication also denoted by $\cdot$.
\begin{Theorem}\label{PROP-well-tempered}
Assume the finite-dimensional unital associative $k$-algebra $A$ and elements $\Delta,\Delta'\in A$ are such that both $A_\Delta$ and $A_{\Delta'}$ have no non-zero idempotents. Let $\Lambda=\left(\begin{smallmatrix} 1\ 0\\ 0\ \Delta\end{smallmatrix}\right),\ \Lambda'=\left(\begin{smallmatrix} 1\ 0\\ 0\ \Delta'\end{smallmatrix}\right)$. Then if $\hat{M}_\Lambda$ and $\hat{M}_{\Lambda'}$ are isomorphic as unital algebras, then $A_\Delta$ and $A_{\Delta'}$ are isomorphic as non-unital algebras.
\end{Theorem}
\begin{proof}
 
 Let $e=\left(\begin{smallmatrix} 1\ 0\\ 0\ 0\end{smallmatrix}\right)$, $\varepsilon=\mathbf{1}-e$, where $\mathbf{1}$ is the adjoined unit of $\hat{M}_\Lambda.$ Note that $\forall X\in M_\Lambda\subset\hat{M}_\Lambda$ we have
 $$e\cdot X=e\Lambda X=eX,\ X\cdot e=X\Lambda e=Xe,$$ $$\varepsilon\cdot X=X-e\cdot
 X=\left(\begin{smallmatrix} 0\ 0\\ 0\ 1\end{smallmatrix}\right)X,\ X\cdot \varepsilon=X-X\cdot
 e=X\left(\begin{smallmatrix} 0\ 0\\ 0\ 1\end{smallmatrix}\right).$$  
 
 A straightforward computation now shows that both $e$ and $\varepsilon$ are orthogonal idempotents,  $\varepsilon\cdot \hat{M}_\Lambda\cdot \varepsilon$ is generated as a $k$-vector space by the elements $\varepsilon$ and $\left(\begin{smallmatrix} 0\ 0\\ 0\ x\end{smallmatrix}\right)$, $x\in A$, and $e\cdot\hat{M}_\Lambda\cdot e$ consists of elements of the form $\left(\begin{smallmatrix} x\ 0\\ 0\ 0\end{smallmatrix}\right)$, $x\in A$.
 
 We also have that $\forall x,y\in A$ 
 $$\left(\begin{smallmatrix} x\ 0\\ 0\ 0\end{smallmatrix}\right)\cdot \left(\begin{smallmatrix} y\ 0\\ 0\ 0\end{smallmatrix}\right)=\left(\begin{smallmatrix} xy\ 0\\ 0\ 0\end{smallmatrix}\right),\ \left(\begin{smallmatrix} 0\ 0\\ 0\ x\end{smallmatrix}\right)\cdot \left(\begin{smallmatrix} 0\ 0\\ 0\ y\end{smallmatrix}\right)=\left(\begin{smallmatrix} 0\ 0\\ 0\ x\Delta y\end{smallmatrix}\right),\ \varepsilon\cdot \left(\begin{smallmatrix} 0\ 0\\ 0\ x\end{smallmatrix}\right)=\left(\begin{smallmatrix} 0\ 0\\ 0\ x\end{smallmatrix}\right)\cdot\varepsilon=\left(\begin{smallmatrix} 0\ 0\\ 0\ x\end{smallmatrix}\right).$$ This computation shows that $\hat{A}_\Delta,$ and $\varepsilon\cdot \hat{M}_\Lambda\cdot \varepsilon$ are isomorphic as unital $k$-algebras,  with the isomorphism $f:\hat{A}_\Delta\to \varepsilon\cdot \hat{M}_\Lambda\cdot \varepsilon$, $f(\mathbf{1})=\mathbf{\varepsilon},$ $f(x)=\left(\begin{smallmatrix} 0\ 0\\ 0\ x\end{smallmatrix}\right),$ where $\mathbf{1}$ is the adjoined unit of $\hat{A}_\Delta.$ On the other hand, $e\cdot\hat{M}_\Lambda\cdot e$ is evidently isomorphic to $A.$

Assume now that we have an isomorphism $\varphi:\hat{M}_\Lambda\to\hat{M}_{\Lambda'},$ $\Lambda'=\left(\begin{smallmatrix} 1\ 0\\ 0\ \Delta'\end{smallmatrix}\right)$, and that $\varepsilon$ is a primitive idempotent both as an element of $\hat{M}_{\Lambda}$ and of $\hat{M}_{\Lambda'}$. Since $$\text{dim } \varepsilon\cdot \hat{M}_\Lambda\cdot \varepsilon=\text{dim }\hat{A}_\Delta=\text{dim }A+1=\text{dim } e\cdot \hat{M}_\Lambda\cdot e+1,$$ and the same is true for $\hat{M}_{\Lambda'}$, $\varphi$ must restrict to an isomorphism of $\varepsilon\cdot \hat{M}_\Lambda\cdot \varepsilon$ and $\varepsilon\cdot \hat{M}_{\Lambda'}\cdot \varepsilon$. This is the same as an isomorphism of $\hat{A}_\Delta$ and $\hat{A}_{\Delta'}.$

It remains to show that under the assumption of the proposition, $\varepsilon$ is indeed a primitive idempotent of both $\hat{M}_{\Lambda}$ and $\hat{M}_{\Lambda'}$. We will demonstrate it for $\hat{M}_\Lambda.$ To see this, assume on the contrary that it is not primitive. It is easy to see then that any decomposition into a sum of idempotents must have the form $\varepsilon=\mathbf{1}-e=\mathbf{1}-X+Y$ with $X,Y$ - idempotents of $M_\Lambda.$ Then $X=Y+e,$ so $(Y+e)^2=Y+e$ and $Y\cdot e+e\cdot Y=0.$ If we write $Y=\left(\begin{smallmatrix} a\ b\\ c\ d\end{smallmatrix}\right),$ we end up with the equality $\left(\begin{smallmatrix} 2a\ b\\ c\ 0\end{smallmatrix}\right)=0,$ so $Y=\left(\begin{smallmatrix} 0\ 0\\ 0\ d\end{smallmatrix}\right),$ and $Y$ being idempotent in $M_\Lambda$ is equivalent to $d$ being idempotent in $A_\Delta,$ which, by the assumption, implies that $d=0.$

\end{proof}
Now we take as $A$ the 16-dimensional algebra $B$ of \autoref{SEC-Associative} and as $\Delta$ we take elements of the family $\Delta(\lambda)=\lambda h_1+h_2,$ $\lambda\in k$. 
\begin{Corollary}
Let $M=Mat_2(B)$, where the algebra $B$ was defined in \autoref{SEC-Associative}, $\Delta(\lambda)=\lambda h_1+h_2$, $\lambda\in k,$ $\Lambda(\lambda)=\left(\begin{smallmatrix} 1\hfill\ 0\hfill\\ 0\hfill\ \Delta(\lambda)\hfill\end{smallmatrix}\right).$ Then the homotopes $M_{\Lambda(\lambda)}, M_{\Lambda(\lambda')}$ are isomorphic if and only if $\lambda'=\lambda^{\pm1}.$ Moreover, algebras $M_{\Lambda(\lambda)}, M_{\Lambda(\lambda')}$ are Morita equivalent if and only if $\lambda'=\lambda^{\pm1}.$
\end{Corollary}

\begin{proof}
According to \autoref{PROP-well-tempered}, we only need to check that $B_{\Delta(\lambda))}$ has no non-zero idempotents $\forall \lambda\in k$. This is the case, because such an idempotent $i$ must satisfy $$i=i\Delta(\lambda) i=\lambda ih_1i+ih_2i$$ in $B$, and there are no non-zero such $i$ in $B$, since either $i\Delta(\lambda)i=0$, or the shortest length of the words appearing in $\lambda ih_1i+ih_2i$ is strictly larger than the shortest length of the summands of $i.$
But the algebras $M_{\Lambda(\lambda)}$ are basic $\forall \lambda$, so for these algebras being isomorphic is equivalent to being Morita equivalent. 
\end{proof}

%As a result, for this algebra an infinite family of pairwise non-isomorphic homotopes $B_{\Delta(\lambda)}$ gives us an infinite family of pairwise non-isomorphic homotopes $\hat{M}_{\Lambda(\lambda)}$, with each $\Lambda(\lambda)$ well-tempered. Note that since all the algebras involved are basic, this also gives us an infinite family of homotopes which are not pairwise Morita equivalent, even though $\Lambda(\lambda)$ are well-tempered.

%\bibliography{ref.bib}

\begin{thebibliography}{30}
\bibitem{Al}
A. A. Albert.  Non-associative algebras:  I. fundamental concepts and isotopy.Annals of Mathematics, 43(4):685–707, 1942.

\bibitem{Aus}
M. Auslander, M. I. Platzeck, and G. Todorov. Homological theory of idempotent ideals. Transactions of the American Mathematical Society, 332(2):667–692, 1992.

\bibitem{ballico_1992}
E. Ballico.  On projective varieties with projectively equivalent zero-dimensional linear sections. Canadian Mathematical Bulletin, 35(1):3–13, 1992.

\bibitem{BZ2}
A. Bondal and I. Zhdanovskiy. Orthogonal pairs and mutually unbiased bases. Journal of Mathematical Sciences, 216(1):23–40, 2016.

\bibitem{BZ1}
A.  I.  Bondal  and  I.  Y.  Zhdanovskiy.   Theory  of  homotopes  with  applications  to  mutually  unbiased  bases,  harmonic  analysis  on  graphs,  and  perverse  sheaves. Russian Mathematical Surveys, 76(2):195, 2021.

\bibitem{Br} R. H. Bruck. Some results in the theory of linear non-associative algebras. Transactions of the American Mathematical Society, 56(2):141–199, 1944.

\bibitem{Eisenbud_linear}
D. Eisenbud.  Linear sections of determinantal varieties. American  Journal  of  Mathematics, 110(3):541–575, 1988.

\bibitem{gabriel1975}
P.  Gabriel.   Finite  representation  type  is  open.   In  V.  Dlab  and  P.  Gabriel,  editors,Representations  of  Algebras, pages 132–155, Berlin, Heidelberg, 1975. Springer Berlin Heidelberg.

\bibitem{HARRIS198471}
J. Harris and L. W. Tu.  On symmetric and skew-symmetric determinantal varieties. Topology, 23(1):71–84, 1984.

\bibitem{Ja}
N. Jacobson. Structure  theory  of  Jordan  algebras,  volume 5.  University of Arkansas, 1981.

\bibitem{KZ1}
A.  Kocherova  and  I.  Y.  Zhdanovskiy.   On  the  algebra  generated  by  projectors  withcommutator relation. Lobachevskii Journal of Mathematics, 38(4):670–687, 2017.

\bibitem{KZ2}
A. Kocherova and I. Y. Zhdanovskiy.  Some algebraic and geometric constructions inquantum information theory. AIP Conference Proceedings, 2362(1), 2021.

\bibitem{Lvovsky_1994}
S. L’vovsky.  On curves and surfaces with projectively equivalent hyperplane sections. Canadian Mathematical Bulletin, 37(3):384–392, 1994.

\bibitem{Ma}
A. I. Mal’tsev.  On a representation of nonassociative rings. Uspekhi Matematicheskikh Nauk, 7(1):181–185, 1952.

\bibitem{McC}
K. McCrimmon. A taste of Jordan algebras.  Springer, 2004.

\bibitem{pardini_remarks_1994}
R. Pardini. Some remarks on varieties with projectively isomorphic hyperplane sections. Geometriae Dedicata, 52(1):15–32, Aug. 1994.

\bibitem{Parfenov_2001}
P.  G.  Parfenov.   Orbits  and  their  closures  in  the  spaces ${\mathbb C}^{k_1} \otimes \dots \otimes {\mathbb C}^{k_r}$. Sbornik: Mathematics, 192(1):89–112, February 2001.

\bibitem{Psa}
C. Psaroudakis.  Homological theory of recollements of abelian categories. Journal of Algebra, 398:63–110, 2014.

\bibitem{Z}
I. Y. Zhdanovskiy.  Homotopes of finite-dimensional algebras. Communications in Algebra, 49(1):43–57, 2021.

\bibitem{KZ3}
I. Y. Zhdanovskiy  and  A. S. Kocherova.   Algebras  of  projectors  and  mutually  unbiased bases in dimension 7. Journal of Mathematical Sciences, 241(2):125–157, 2019

\end{thebibliography}
%\bibliographystyle{abbrv}

\end{document}